\documentclass[a4paper,reqno,11pt]{amsart}

\usepackage{amsmath}
\usepackage{amssymb}
\usepackage{amsthm}
\usepackage{latexsym}
\usepackage[all]{xypic}
\usepackage{amsfonts}
\usepackage{mathrsfs}
\usepackage{graphicx,color}
\usepackage{xcolor}
\usepackage{tikz}
\usepackage{graphics}
\usetikzlibrary[shapes]
\usepackage{anysize,hyperref}
\usepackage{soul}
\setstcolor{red}

\newtheorem{theorem}{Theorem}[section]
\newtheorem{lemma}[theorem]{Lemma}

\newtheorem{proposition}[theorem]{Proposition}
\theoremstyle{definition}
\newtheorem{definition}[theorem]{Definition}
\newtheorem{definition-proposition}[theorem]{Definition-Proposition}
\newtheorem{remark}[theorem]{Remark}
\newtheorem{example}[theorem]{Example}

\def\C{\mathcal{C}}
\def\D{\mathcal{D}}
\def\H{\mathcal {H}}

\def\W{\mathscr{W}}

\def\X{\mathscr{X}}
\def\Y{\mathscr{Y}}
\def\Z{\mathcal {Z}}
\def\I{\mathcal {I}}

\def\M{\mathcal{M}}

\def \text{\mbox}

\providecommand{\add}{\mathop{\rm add}\nolimits}%
\providecommand{\Ext}{\mathop{\rm Ext}\nolimits}%
\providecommand{\Hom}{\mathop{\rm Hom}\nolimits}%
\renewcommand{\mod}{\mathop{\rm mod}\nolimits}%

\def\XX{\widetilde{\X}}
\def\YY{\widetilde{\Y}}

\providecommand{\nc}{\mathop{\rm nc}\nolimits}%

\usetikzlibrary{arrows,decorations.pathmorphing,decorations.pathreplacing,positioning,shapes.geometric,shapes.misc,decorations.markings,decorations.fractals,calc,patterns}
\begin{document}

\title{Mutation of $n$-cotorsion pairs in triangulated categories}

\author[Chang]{Huimin Chang}
\address{
Department of Applied Mathematics,
The Open University of China,
100039 Beijing,
P. R. China
}
\email{changhm@ouchn.edu.cn}
\author[Zhou]{Panyue Zhou}
\address{School of Mathematics and Statistics, Changsha University of Science and Technology, 410114 Changsha, Hunan,  P. R. China}
\email{panyuezhou@163.com}
\begin{abstract}
In this article, we define the notion of $n$-cotorsion pairs in triangulated categories, which is a generalization of the classical cotorsion pairs. We prove that any mutation of an $n$-cotorsion pair is again an $n$-cotorsion pair. When $n=1$, this result generalizes the work of Zhou and Zhu for classical cotorsion pairs. As applications, we give a geometric characterization of $n$-cotorsion pairs in $n$-cluster categories of type $A$ and give a geometric realization of mutation of $n$-cotorsion pairs via rotation of certain configurations of $n$-diagonals.
\end{abstract}



\thanks{Huimin Chang is supported by the National Natural Science Foundation of China (Grant No. 12301047).
Panyue Zhou is supported by the National Natural Science Foundation of China (Grant No. 12371034) and by the Hunan Provincial Natural Science Foundation of China (Grant No. 2023JJ30008). }

\maketitle

\section{Introduction}
The notion of torsion pairs in abelian categories was first introduced by Dickson \cite{D}, and the triangulated version was studied by Iyama and Yoshino \cite{IY}. Later, cotorsion pairs in a triangulated category were introduced by Nakaoka \cite{N} to unify the abelian structures arising from $t$-structures \cite{BBD} and from cluster tilting subcategories \cite{KR,KZ,IY}. Torsion pairs and cotorsion pairs in a triangulated category can be transformed into each other by shifting the torsion-free parts. Hence, it is equivalent to consider torsion pairs and cotorsion pairs in triangulated categories.

 Let $n$ be a positive integer. Motivated by some properties satisfied by Gorenstein projective and Gorenstein injective modules over an Iwanaga-Gorenstein ring, Huerta, Mendoza and Pérez \cite[Definition 2.2]{HMP} introduced the notion of $n$-cotorsion (resp. left $n$-cotorsion, right $n$-cotorsion) pairs in an abelian categories. It is worth to note that $1$-cotorsion pairs coincide with the concept of complete cotorsion pairs. Recently, He and Zhou \cite{HZ} introduced $n$-cotorsion (resp. left $n$-cotorsion, right $n$-cotorsion) pairs in extriangulated categories, which introduced by Nakaoka and Palu in their seminal work \cite{NP}. Extriangulated categories simultaneously generalizes exact categories and triangulated categories. Abelian categories and extension-closed subcategories of a triangulated category are considered as specific instances of extriangulated categories.  Especially, when $n=1$, an $n$-cotorsion pair is just a cotorsion pair in the sense of Nakaoka and Palu \cite[Defintion 4.1]{NP}.

An important motivation for mutation comes from cluster algebras. Cluster algebras were introduced by Fomin and Zelevinsky \cite{FZ} in order to give an algebraic and combinatorial framework for the positive and canonical basis of quantumn groups. The mutation of clusters was defined in cluster algebras. As a categorization of cluster algebra, Buan et al \cite{BMRRT} introduced cluster categories. Cluster categories are triangulated categories by Keller\cite{K}. The rigid indecomposable objects in cluster categories correspond to cluster variables and cluster tilting subcategories correspond to clusters. Just like mutation of cluster, mutation of cluster tilting subcategory was studied through replacing one indecomposable object by a unique other indecomposable object such that one gets a cluster tilting subcategory again. See \cite{BMRRT} for details. As a generalization, mutation of cluster tilting subcategories in arbitrary Krull-Schmidt $K$-linear triangulated categories was studied in \cite{BIRS,IY,P}. Iyama and Yoshino \cite{IY} introduced a more general concept mutation in triangulated categories that makes mutation of cluster tilting subcategory a special case. Zhou and Zhu \cite{ZZ} proved that any mutation of a torsion pair is again a torsion pair.

In this paper, we define $n$-cotorsion pairs in triangulated categories, which is a generalization of the classical cotorsion pairs. It should be noted that the notion of $n$-cotorsion pairs in triangulated categories we defined in this paper is different from that defined by He and Zhou in extriangulated categories and deduced to triangulated categories. We prove that any mutation of an $n$-cotorsion pair is again an $n$-cotorsion pair. When $n=1$, this result generalizes a work of Zhou and Zhu for classical cotorsion pairs. As applications, we give a geometric characterization of $n$-cotorsion pairs in $n$-cluster categories of type $A$ and give a geometric realization of mutation of $n$-cotorsion pairs via rotation of certain configurations of $n$-diagonals.

This paper is organized as follows. In Section 2 we give an overview of cotorsion pairs in triangulated categories, and  mutation of subcategories of triangulated categories.  In Section 3 we define $n$-cotorsion pairs in triangulated categories, recall from \cite{IY} the construction of subfactor triangulated categories and study its compatibility with $n$-cotorsion pairs, and prove the main result.  In Section 4 we give a geometric characterization of $n$-cotorsion pairs in $n$-cluster categories of type $A$ via certain configurations of $n$-diagonals, and introduce the rotation of such combinatorial models, which can give a geometric realization of mutation of $n$-cotorsion pairs in $n$-cluster categories of type $A$.

\subsection*{Conventions} In this paper, $K$ stands for an algebraically closed field. All additive categories considered  are assumed to be Krull-Schmidt, i.e. any object is isomorphic to a finite direct sum of objects whose endomorphism rings are local.
Let $\C$ be a triangulated category, we denote by $\Hom_{\C}(X, Y)$ the set of morphisms from $X$ to $Y$ in $\C$. We denote the
composition of $f\in\Hom_{\C}(X, Y)$ and $g\in\Hom_{\C}(Y, Z)$ by $g\circ f\in\Hom_{\C}(X, Z)$.
When we say that $\X$ is a subcategory of $\C$, we always mean that $\C$
is a full subcategory which is closed under isomorphisms, direct sums and
direct summands.
We denote by $\X^\perp$ (resp. $^\perp\X$) the subcategory whose objects are $M\in\C$ satisfying $\Hom_{\C}(\X,M)=0$ (resp. $\Hom_{\C}(M,\X)=0$). We use $\Ext_{\C}^i(X,Y)$ to denote $\Hom_{\C}(X,Y[i])$, $i\in\mathbb{Z}$, where $[1]$ is the shift functor of $\C$. For an object $X\in\C$, $\add X$ means the additive closure of $X$.
Let $\X$ and $\Y$ be subcategories of $\C$.
We denote by $\X\ast\Y$ the collection of objects in $\C$
consisting of all such $C\in\C$ with triangles
$$X\to C\to Y\to X[1]$$
where $X\in\X$ and $Y\in\Y$.

\section{Preliminaries}
We recall the definition of cotorsion pairs in triangulated categories, and  mutation of subcategories of triangulated categories.

\subsection{Cotorsion pairs in triangulated categories}
We briefly review the definition of the classical cotorsion pairs and some related results in triangulated categories from \cite{IY,CZZ,ZZ}.
\begin{definition}
Let $\X$ and $\Y$ be subcategories of a triangulated category $\C$.
\begin{itemize}
  \item [(1)] The pair $(\X, \Y)$ is called a torsion pair \cite{IY} if
  $$\Hom_{\C}(\X, \Y)=0\;\text{and}\;\C=\X\ast\Y. $$
  The subcategory $\I(\X)=\X\cap\Y[-1]$ is called the core of the torsion pair.
  \item [(2)] The pair $(\X, \Y)$ is called a cotorsion pair \cite{N} if
  $$\Ext^{1}_{\C}(\X, \Y)=0\;\text{and}\;\C=\X\ast\Y[1]. $$
  The subcategory $\I(\X)=\X\cap\Y$ is called the core of the cotorsion pair.
\end{itemize}
\end{definition}

\begin{remark}\label{remark3}
By definition, we know that $(\X,\Y)$ is a cotorsion pair if and only if $(\X,\Y[1])$ is a torsion pair.
\end{remark}

\begin{definition}
Let $\X$ be a subcategory of a triangulated category $\C$ and $n\geq 1$ be an integer.
\begin{itemize}
  \item [(1)] We call $\X$ an $(n+1)$-rigid subcategory if $\Ext^i_{\C}(\X,\X)=0$ for $1\leq i\leq n$.
  \item [(2)] For an object $C$ in $\C$, a right $\X$-approximation of $C$ in $\C$ is a map $X\rightarrow C$, with $X$ in $\X$, such that for all objects $Y$ in $\X$, the sequence
$$\Hom_\C(Y,X)\rightarrow\Hom_\C(Y,C)\rightarrow 0$$
is exact. Dually, we have the concept of left $\X$-approximation of $C$.
  \item [(3)] We call $\X$ contravariantly finite if  any object $C\in\C$ admits a right $\X$-approximation, $\X$  covariantly finite if any object $C\in\C$ admits a left $\X$-approximation, and $\X$  functorially finite if both $\X$ contravariantly finite and covariantly finite.
  \item [(4)] We call $\X$ an $(n+1)$-cluster tilting subcategory if $\X$ is functorially finite and satisfying
  $$\X=\bigcap\limits_{i=1}^{n}\X[-i]^\perp=\bigcap\limits_{i=1}^{n}{^\bot}\X[i].$$
\end{itemize}
\end{definition}

Note that when $n=1$, 2-rigid subcategory is simply called rigid and 2-cluster tilting subcategory is  called cluster tilting.

\begin{lemma}\label{m} Let $\X$ and $\Y$ be subcategories of a triangulated category $\C$.
\begin{itemize}
\item [(1)] The pair $(\X, \Y)$ is  a cotorsion pair if and only if the following hold.
\begin{itemize}
\item [(a)] $\X{^\bot}[-1]=\Y$;
\item [(b)] $\X={^\bot}\Y[1]$;
\item [(c)] $\X$ is contravariantly finite or $\Y$ is covariantly finite.
\end{itemize}
\item [(2)] $(\X, \X)$ is  a cotorsion pair if and only if $\X$ is cluster tilting.
\item [(3)] The core $\I(\X)=\X\cap\Y$ of the cotorsion pair $(\X, \Y)$ is rigid.
\end{itemize}
\end{lemma}

\subsection{Mutation of subcategories of triangulated categories}
In this subsection, we review from \cite{IY} the notion of mutation of subcategories of triangulated categories.

The following result is well known and straightforward to check.
\begin{lemma}\cite[Lemma 1.1]{BM2}
Let $\X$ be a subcategory of $\C$ and $C$ be an object of $\C$.
\begin{itemize}
  \item [(1)] If there is a right $\X$-approximation of $C$, then there is a minimal right $\X$-approximation of $C$, unique up to isomorphism.
  \item [(2)] If $f: X\rightarrow C$ is a minimal right $\X$-approximation of $C$, then each right $\X$-approximation is, up to isomorphism, of the form $f\oplus 0: X\oplus X^{'}\rightarrow C$.
\end{itemize}
\end{lemma}

\begin{definition}\cite[Definition 2.5]{IY}\label{h}
Fix a functorially finite rigid subcategory $\D$ of $\C$. For a subcategory $\X$ of $\C$, put
$$\mu^{-1}_{\D}(\X):=(\D\ast\X[1])\cap{^\bot}\D[1].$$
That is, $\mu^{-1}_{\D}(\X)$ consists of all $M\in\C$ such that there exists a triangle
$$X\stackrel{f}\rightarrow D\rightarrow M\rightarrow X[1]$$
with $X\in\X$ and a left $\D$-approximation $f$. Dually, for a subcategory $\Y$ of $\C$, put
$$\mu_{\D}(\Y):=(\Y[-1]\ast\D)\cap\D[-1]^{\bot}.$$
That is, $\mu_{\D}(\Y)$ consists of all $M\in\C$ such that there exists a triangle
$$M\rightarrow D\stackrel{g} \rightarrow Y\rightarrow M$$
with  $Y\in\Y$ and a right $\D$-approximation  $g$.

 In this case, $\mu^{-1}_{\D}(\X)$ is called the forward $\D$-mutation of $\X$ and $\mu_{\D}(\Y)$ is called the backward $\D$-mutation of $\Y$.
\end{definition}

\begin{remark}
It is clear that $\mu_{\D}(\D)=\D=\mu^{-1}_{\D}(\D)$. When $\D=0$, we have $\mu^{-1}_{\D}(\M)=\M[1]$ and $\mu_{\D}(\M)=\M[-1]$.
\end{remark}
\section{Mutation of $n$-cotorsion pairs}
In this section, we fix an integer $n\geq 1$.

\subsection{$n$-cotorsion pairs in triangulated categories}

We define $n$-cotorsion pairs in  a triangulated category, which is a generalization of the classical cotorsion pairs.
\begin{definition}
Let $\C$ be a triangulated category and $\X,\Y$ be two subcategories of $\C$. The pair $(\X,\Y)$ is called an $n$-cotorsion pair if the following conditions hold.
\begin{enumerate}
  \item [(1)] $\X=\bigcap\limits_{i=1}^{n}{^\bot}\Y[i]$~~ for all $1\leq i\leq n$;
  \item [(2)] $\Y=\bigcap\limits_{i=1}^{n}\X[-i]^\perp$~~ for all $1\leq i\leq n$;
   \item [(3)] $\X$ is contravariantly finite and $\Y$ is covariantly finite.
\end{enumerate}
\end{definition}

\begin{remark}
It is easy to check that the pair $(\X,\X)$ is an $n$-cotorsion pair in $\C$ if and only if $\X$ is $(n+1)$-cluster tilting. When $n=1$, the concept of $1$-cotorsion pair is compatible with the classical definition of a cotorsion pair in the sense of Nakaoka \cite{N}.
\end{remark}

\begin{remark}
In \cite{HZ}, He and Zhou defined the notion of $n$-cotorsion pairs in extriangulated categories. Note that exact categories and triangulated categories are extriangulated categories. However, when their concept degenerates into triangulated categories, it does not align with our defined $n$-cotorsion pairs. In particular, in the case of finite triangulated categories, i.e. there are only finitely many indecomposable objects up to isomorphisms,  their $n$-cotorsion pairs satisfy our defined $n$-cotorsion pairs. Therefore, we also adopt the term $n$-cotorsion pair.
\end{remark}

 Now we give some examples of $n$-cotorsion pairs.

\begin{example}
\begin{enumerate}
  \item [(1)] Let $\C_{A_{3}}^{2}$ be the 2-cluster category of type $A_3$. We give an example of a $2$-cotorsion pair in $\C_{A_{3}}^{2}$, see Example \ref{z} for details.

  \item [(2)] Let $\C_{D_{4}}^{3}$ be the 3-cluster category of type $D_4$. The AR-quiver of $\C_{D_{4}}^{3}$ is shown in Figure \ref{aa}. Let $\X=\{(1,5),(1,8),(4,8),(3,7),(2,6)\}$ and $\Y=\{(1,8),(1,11)_r,(1,11)_g,(10,20)_r,(10,20)_g,(9,19)_r,(9,19)_g,(8,18)_r,(8,18)_g\}$ be two subcategories of $\C_{D_{4}}^{3}$. One can check directly from the AR-quiver of $\C_{D_{4}}^{3}$ that $\X=\bigcap\limits_{i=1}^{3}{^\bot}\Y[i]$ and $\Y=\bigcap\limits_{i=1}^{3}\X[-i]^\perp$. Thus, $(\X,\Y)$ is a $3$-cotorsion pair.
\begin{center}

\begin{figure}[h]
\begin{tikzpicture}[scale=0.8,
fl/.style={->,shorten <=6pt, shorten >=6pt,>=latex}]
%%%%%%%%%%%%%%%%%%%%%
% Coordinates
%%%%%%%%%%%%%%%%%%%%%
\coordinate (13) at (0,0) ;
\coordinate (14) at (1,1) ;
\coordinate (15) at (2,2) ;
\coordinate (00) at (2,1) ;
\coordinate (11) at (4,1) ;
\coordinate (22) at (6,1) ;
\coordinate (33) at (8,1) ;
\coordinate (44) at (10,1) ;
\coordinate (55) at (12,1) ;
\coordinate (66) at (14,1) ;
\coordinate (77) at (16,1) ;
\coordinate (88) at (18,1) ;
\coordinate (99) at (20,1) ;
\coordinate (1010) at (22,1) ;
\coordinate (24) at (2,0) ;
\coordinate (25) at (3,1) ;
\coordinate (26) at (4,2) ;
\coordinate (35) at (4,0) ;
\coordinate (36) at (5,1) ;
\coordinate (37) at (6,2) ;
\coordinate (46) at (6,0) ;
\coordinate (47) at (7,1) ;
\coordinate (48) at (8,2) ;
\coordinate (57) at (8,0) ;
\coordinate (58) at (9,1) ;
\coordinate (59) at (10,2) ;
\coordinate (68) at (10,0) ;
\coordinate (69) at (11,1) ;
\coordinate (610) at (12,2) ;
\coordinate (79) at (12,0) ;
\coordinate (710) at (13,1) ;
\coordinate (711) at (14,2) ;
\coordinate (810) at (14,0) ;
\coordinate (811) at (15,1) ;
\coordinate (812) at (16,2) ;
\coordinate (911) at (16,0) ;
\coordinate (912) at (17,1) ;
\coordinate (913) at (18,2) ;
\coordinate (1012) at (18,0) ;
\coordinate (1013) at (19,1) ;
\coordinate (1014) at (20,2) ;
\coordinate (1113) at (20,0) ;
\coordinate (1114) at (21,1) ;
\coordinate (1115) at (22,2) ;
%%%%%%%%%%%%%%%%%%%%%
% Arrows
%%%%%%%%%%%%%%%%%%%%%

\draw[fl] (13) -- (14) ;
\draw[fl] (14) -- (15) ;
\draw[fl] (14) -- (24) ;
\draw[fl] (14) -- (00) ;
\draw[fl] (00) -- (25) ;
\draw[fl] (25) -- (11) ;
\draw[fl] (11) -- (36) ;
\draw[fl] (36) -- (22) ;
\draw[fl] (22) -- (47) ;
\draw[fl] (47) -- (33) ;
\draw[fl] (33) -- (58) ;
\draw[fl] (58) -- (44) ;
\draw[fl] (44) -- (69) ;
\draw[fl] (69) -- (55) ;
\draw[fl] (55) -- (710) ;
\draw[fl] (710) -- (66) ;
\draw[fl] (66) -- (811) ;
\draw[fl] (811) -- (77) ;
\draw[fl] (77) -- (912) ;
\draw[fl] (912) -- (88) ;
\draw[fl] (88) -- (1013) ;
\draw[fl] (1013) -- (99) ;
\draw[fl] (99) -- (1114) ;
\draw[fl] (1114) -- (1010) ;
\draw[fl] (15) --(25) ;
\draw[fl] (24) --(25) ;
\draw[fl] (25) --(35) ;
\draw[fl] (25) --(26) ;
\draw[fl] (35) --(36) ;
\draw[fl] (36) --(37) ;
\draw[fl] (26) --(36) ;
\draw[fl] (46) --(47) ;
\draw[fl] (47) --(48) ;
\draw[fl] (37) --(47) ;
\draw[fl] (36) --(46) ;
\draw[fl] (57) --(58) ;
\draw[fl] (58) --(59) ;
\draw[fl] (48) --(58) ;
\draw[fl] (47) --(57) ;
\draw[fl] (58) --(68) ;
\draw[fl] (68) --(69) ;
\draw[fl] (69) --(610) ;
\draw[fl] (59) --(69) ;
\draw[fl] (79) --(710) ;
\draw[fl] (710) --(711) ;
\draw[fl] (610) --(710) ;
\draw[fl] (69) --(79) ;
\draw[fl] (810) --(811) ;
\draw[fl] (811) --(812) ;
\draw[fl] (711) --(811) ;
\draw[fl] (710) --(810) ;
\draw[fl] (911) --(912) ;
\draw[fl] (912) --(913) ;
\draw[fl] (812) --(912) ;
\draw[fl] (811) --(911) ;
\draw[fl] (1012) --(1013) ;
\draw[fl] (1013) --(1014) ;
\draw[fl] (913) --(1013) ;
\draw[fl] (912) --(1012) ;
\draw[fl] (1113) --(1114) ;
\draw[fl] (1114) --(1115) ;
\draw[fl] (1014) --(1114) ;
\draw[fl] (1013) --(1113) ;
\draw (13) node[scale=0.5] {(1,5)} ;
\draw (14) node[scale=0.5] {(1,8)} ;
\draw (15) node[scale=0.5] {$(1,11)_r$} ;
\draw (00) node[scale=0.5] {$(1,11)_g$} ;
\draw (11) node[scale=0.5] {$(4,14)_r$} ;
\draw (22) node[scale=0.5] {$(7,17)_g$} ;
\draw (33) node[scale=0.5] {$(10,20)_r$} ;
\draw (44) node[scale=0.5] {$(3,13)_g$} ;
\draw (55) node[scale=0.5] {$(6,16)_r$} ;
\draw (66) node[scale=0.5] {$(9,19)_g$} ;
\draw (77) node[scale=0.5] {$(2,12)_r$} ;
\draw (88) node[scale=0.5] {$(5,15)_g$} ;
\draw (99) node[scale=0.5] {$(8,18)_r$} ;
\draw (1010) node[scale=0.5] {$(1,11)_g$} ;
\draw (24) node[scale=0.5] {(4,8)} ;
\draw (25) node[scale=0.5] {(4,11)} ;
\draw (26) node[scale=0.5] {$(4,14)_g$} ;
\draw (35) node[scale=0.5] {(7,11)} ;
\draw (36) node[scale=0.5] {(7,14)} ;
\draw (37) node[scale=0.5] {$(7,17)_r$} ;
\draw (46) node[scale=0.5] {(10,14)} ;
\draw (47) node[scale=0.5] {(10,17)} ;
\draw (48) node[scale=0.5] {$(10,20)_g$} ;
\draw (57) node[scale=0.5] {(3,7)} ;
\draw (58) node[scale=0.5] {(3,10)} ;
\draw (59) node[scale=0.5] {$(3,13)_r$} ;
\draw (68) node[scale=0.5] {(6,10)} ;
\draw (69) node[scale=0.5] {(6,13)} ;
\draw (610) node[scale=0.5] {$(6,16)_g$} ;
\draw (79) node[scale=0.5] {(9,13)} ;
\draw (710) node[scale=0.5] {(9,16)} ;
\draw (711) node[scale=0.5] {$(9,19)_r$} ;
\draw (810) node[scale=0.5] {(2,6)} ;
\draw (811) node[scale=0.5] {(2,9)} ;
\draw (812) node[scale=0.5] {$(2,12)_g$} ;
\draw (911) node[scale=0.5] {(5,9)} ;
\draw (912) node[scale=0.5] {(5,12)} ;
\draw (913) node[scale=0.5] {$(5,15)_r$} ;
\draw (1012) node[scale=0.5] {(8,12)} ;
\draw (1013) node[scale=0.5] {(8,15)} ;
\draw (1014) node[scale=0.5] {$(8,18)_g$} ;
\draw (1113) node[scale=0.5] {(1,5)} ;
\draw (1114) node[scale=0.5] {(1,8)} ;
\draw (1115) node[scale=0.5] {$(1,11)_r$} ;
\end{tikzpicture}
\caption{The AR-quiver of  $\C_{D_{4}}^{3}$}
\label{aa}
\end{figure}
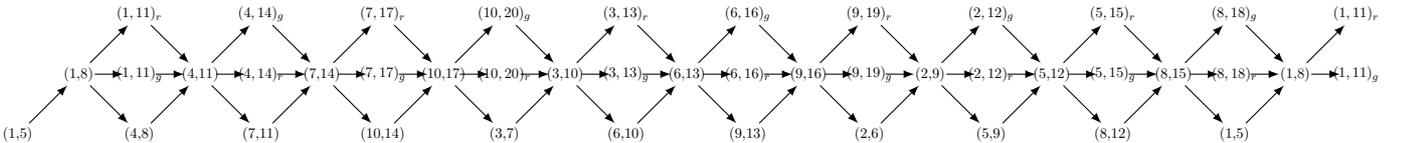
\end{center}

\end{enumerate}
\end{example}

\begin{definition}
Let $(\X,\Y)$ be an $n$-cotorsion pair in $\C$. We call $\I(\X)=\X\cap\Y$ the core of $(\X,\Y)$.
\end{definition}

\begin{remark}\label{l}
By the definition of $n$-cotorsion pair, it is easy to see the core $\I(\X)$ is an $(n+1)$-rigid subcategory.
\end{remark}

\subsection{Compatibility with subfactor triangulated categories}
Let $\C$ be a Hom-finite Krull-Schmidt triangulated category over $K$ with a Serre functor $\mathbb{S}$ from now on. Recall from \cite{BK} that a functor $\mathbb{S}\colon \C\rightarrow \C$ is called a Serre functor if there is a  bifunctorial isomorphism
$$D\Hom_{\C}(X,Y)\cong\Hom_{\C}(Y,\mathbb{S} X),$$
for any objects $X,Y\in\C$.  We put
$$\mathbb{S}_n=\mathbb{S}\circ[-n]:\C\rightarrow\C.$$
Recall that $\C$ is $n$-Calabi-Yau if and only if $\mathbb{S}_n$ is an identity functor.

In this subsection, we fix a functorially finite $(n+1)$-rigid subcategory $\D$ of $\C$ which satisfying the condition $\mathbb{S}\D=\D[n+1]$. Then it is easy to check that $\bigcap\limits_{i=1}^{n}\D[-i]^\perp=\bigcap\limits_{i=1}^{n}{^\bot}\D[i]$, which is denoted by $\Z$. The quotient category $\mathfrak{U}:=\Z/\D$ is called a subfactor triangulated category which is defined by the following data.
\begin{itemize}
  \item [(1)] The objects in $\mathfrak{U}$ are the same as  $\Z$.
  \item [(2)] The morphism space $\Hom_{\mathfrak{U}}(X,Y)$ is defined as
  $$\Hom_{\mathfrak{U}}(X,Y):=\Hom_{\Z}(X,Y)/\D(X,Y)$$
  for each $X,Y\in\Z$, where $\D(X,Y)$ is the subspace of $\Hom_{\Z}(X,Y)$ consisting of morphisms factoring through objects in $\D$.
\end{itemize}

It is proved in \cite{IY} that $\mathfrak{U}$ carries a natural triangulated structure inherited from the triangulated structure of $\C$ as follows.
\begin{itemize}
  \item[$\bullet$] For any object $X\in\mathfrak{U}$, choose a left $\D$-approximation $f:D\rightarrow X$ and extend it to a triangle
      $$X\stackrel{f}\rightarrow D\rightarrow Z\rightarrow X[1],$$
      then the shift of $X$ in $\mathfrak{U}$ is defined to be $Z$, denoted by $X\langle1\rangle$.
  \item[$\bullet$] For any triangle $X\stackrel{a}\rightarrow Y\stackrel{b}\rightarrow Z\stackrel{c}\rightarrow X[1]$ in $\C$ with $X,Y,Z\in\Z$, since $\Hom_{\C}(Z,D[1])=0$, there is a commutative diagram:

$$
{
\xymatrix@-2mm@C-0.01cm{
     X \ar@{=}[d]\ar[rr]^{a} && Y \ar[d]\ar[rr]^{b}&& Z \ar[d]^{d}\ar[rr]^{c} && X[1]\ar@{=}[d] \\
  X \ar[rr]^{f} && D \ar[rr]^{g}&& X\langle1\rangle \ar[rr]^{h} && X[1]\\
\\
}
}
$$
Now we consider the complex
$$X\stackrel{\bar{a}}\rightarrow Y\stackrel{\bar{b}}\rightarrow Z\stackrel{\bar{d}}\rightarrow X\langle1\rangle$$
in $\mathfrak{U}$. We define triangles in $\mathfrak{U}$ as the complexes obtained in this way.
\end{itemize}
The following results are useful.
\begin{lemma}\label{c}
Let $\D$ be a functorially finite $(n+1)$-rigid subcategory of $\C$  satisfying  $\mathbb{S}\D=\D[n+1]$ and $\Z=\bigcap\limits_{i=1}^{n}\D[-i]^\perp=\bigcap\limits_{i=1}^{n}{^\bot}\D[i]$. We have the following results.
\begin{itemize}
  \item [(1)] \cite[Theorem 4.7]{IY} The subfactor category $\mathfrak{U}:=\Z/\D$ forms a triangulated category with a Serre functor
$\mathbb{S}_n\circ\langle n+1\rangle$. In particular, if $\C$ is $(n+1)$-Calabi–Yau, then so is $\mathfrak{U}$.
  \item [(2)]\cite[Lemma 4.8]{IY} For any $X,Y\in\Z$, there exists an isomorphism
  $$\mathfrak{U}(X,Y\langle i\rangle)\cong\C(X,Y[i])$$
  for any $1\leq i\leq n$.
  \item [(3)] \cite[Theorem 4.9]{IY} There exists a one-one correspondence between $(n+1)$-cluster tilting subcategories of $\C$ containing $\D$ and $(n+1)$-cluster tilting subcategories of $\mathfrak{U}$.
\end{itemize}
\end{lemma}
The following lemma is useful.
\begin{lemma}\label{b}
Let $(\X,\Y)$ be an $n$-cotorsion pair in $\C$ with core $\I(\X)$. Then $\D\subset\I(\X)$ if and only if $\D\subset\X\subset\Z$ and $\D\subset\Y\subset\Z$.
\end{lemma}
\begin{proof}
The `if' part is trivial since  $\I(\X)=\X\cap\Y$. To prove the `only if' part, let $X$ be an arbitrary object in $\X$. Since  $\D\subset\I(\X)\subset\Y$, we have $\Hom(X,\D[i])=0$  for $1\leq i\leq n$. It follows that $X\in\Z$ and hence $\X\subset\Z$. Similarly, we have $\Y\subset\Z$.
\end{proof}
By the lemma above, one can consider the relationship between $n$-cotorsion pairs in $\C$ whose cores contain $\D$ and
$n$-cotorsion pairs in the subfactor triangulated category $\mathfrak{U}$. In the following, for a subcategory $\W$ of $\C$ satisfying  $\D\subset\W\subset\Z$, we denote $\overline{\W}$ the subcategory of $\Z$. It is clear that any subcategory of $\Z$ has this form.
\begin{lemma}\label{a}
Let $\X$ be a subcategory of $\C$ satisfying  $\D\subset\X\subset\Z$. Then
\begin{itemize}
  \item [(1)] $\X$ is a contravariantly finite subcategory of $\C$ if and only if $\overline{\X}$ is a contravariantly finite subcategory of $\mathfrak{U}$.
  \item [(2)] $\X$ is a covariantly finite subcategory of $\C$ if and only if $\overline{\X}$ is a covariantly finite subcategory of $\mathfrak{U}$.
  \item [(3)] $\X$ is a functorially finite subcategory of $\C$ if and only if $\overline{\X}$ is a functorially finite subcategory of $\mathfrak{U}$.
\end{itemize}
\end{lemma}
\begin{proof}
We only prove the first assertion, the others can be proved similarly.

Suppose $\X$ is a contravariantly finite subcategory of $\C$. Let $Z$ be an object of $\Z$. Since $\X$ is contravariantly finite, there exists an object $X\in\X$ such that $ X\stackrel{f} \rightarrow Z$ is a right $\X$-approximation of $Z$ in $\C$. It is easy to check that $ X\stackrel{\overline{f}} \rightarrow Z$ is a right $\overline{\X}$-approximation of $Z$ in $\mathfrak{U}$.

Now suppose $\overline{\X}$ is a contravariantly finite subcategory of $\mathfrak{U}$. Since $\D$ is a functorially finite subcategory of $\C$, so is $\Z$ by Proposition 2.4 in \cite{IY} and its dual. Let $C$ be an object of $\C$. Since $\Z$ is a functorially finite subcategory of $\C$, there exists an object $Z\in\Z$ such that $Z\stackrel{f} \rightarrow C$ is a right $\Z$-approximation of $C$ in $\C$. By assumption, there exists an object $X\in\X$ such that $ X\stackrel{\overline{g}} \rightarrow Z$ is a right $\overline{\X}$-approximation of $Z$ in $\mathfrak{U}$. For any morphism $X^\prime\stackrel{h} \rightarrow C$ with $X^\prime\in\X\subset\Z$, since $Z\stackrel{f} \rightarrow C$ is a right $\Z$-approximation of $C$ in $\C$, there exists a morphism $X^\prime\stackrel{\ell}\rightarrow Z$ such that $h=f\circ\ell$. Since $X\stackrel{\overline{g}} \rightarrow Z$ is a right $\overline{\X}$-approximation of $Z$ in $\mathfrak{U}$,
there exists a morphism $X^\prime\stackrel{\overline{k}}\rightarrow X$ such that $\overline{l}=\overline{g}\circ\overline{k}$, i.e. $\overline{l}-\overline{g}\circ\overline{k}=\overline{0}$. So there exists an object $D\in\D$ and two morphisms $X^\prime\stackrel{a_1}\rightarrow D$ and $D\stackrel{a_2}\rightarrow Z$ such that $\ell=g\circ k+a_2\circ a_1$. We will show that $X\oplus D\xrightarrow{(f\circ g,f\circ a_2)} C$ is a right $\X$-approximation of $C$ in $\C$ and thus complete the proof. For any morphism $X^\prime\stackrel{h} \rightarrow C$ with $X^\prime\in\X$, we have found a morphism $X^\prime\xrightarrow{\binom{k}{a_1}}X \oplus D$ such that $h=f\circ\ell=f\circ(g\circ k+a_2\circ a_1)=(f\circ g,f\circ a_2)\circ\binom{k}{a_1}$, so $X\oplus D\xrightarrow{(f\circ g,f\circ a_2)} C$ is a right $\X$-approximation of $C$ in $\C$.
\end{proof}
\begin{theorem}\label{d}
Let $(\X,\Y)$ be an $n$-cotorsion pair in $\C$ with $\D\subset\I(\X)$. Then $(\overline{\X},\overline{\Y})$ is an $n$-cotorsion pair in $\mathfrak{U}$ with $\I(\overline{\X})=\overline{\I(\X)}$. Moreover, the map $(\X,\Y)\mapsto(\overline{\X},\overline{\Y})$ is a bijection from the set of $n$-cotorsion pair in $\C$ whose cores contain $\D$ to the set of $n$-cotorsion pair in $\mathfrak{U}$.
\end{theorem}
\begin{proof}
Since $(\X,\Y)$ is an $n$-cotorsion pair in $\C$ with $\D\subset\I(\X)$, we have $\D\subset\X\subset\Z$ and $\D\subset\Y\subset\Z$ by Lemma \ref{b}. So $\overline{\X}$ and $\overline{\Y}$ are subcategories of $\mathfrak{U}$. By Lemma \ref{c}, there is an isomorphism $\mathfrak{U}(\X,\Y\langle i\rangle)\cong\C(\X,\Y[i])$ for any $1\leq i\leq n$, so $\X=\bigcap\limits_{i=1}^{n}{^\bot}\Y[i]$ for all $1\leq i\leq n$ if and only if $\overline{\X}=\bigcap\limits_{i=1}^{n}{^\bot}\overline{\Y}\langle i\rangle$ for all $1\leq i\leq n$. Moreover, By Lemma \ref{a}, we get $(\overline{\X},\overline{\Y})$ is an $n$-cotorsion pair in $\mathfrak{U}$. In this case, we have  $\I(\overline{\X})=\overline{\X}\cap\overline{\Y}=\overline{\X\cap\Y}=\overline{\I(\X)}$.

The bijection from the set of $n$-cotorsion pair in $\C$ whose cores contain $\D$ to the set of $n$-cotorsion pair in $\Z$ is obvious by Lemma \ref{b}, Lemma \ref{c} and Lemma \ref{a}.
\end{proof}

\subsection{Mutation of $n$-cotorsion pairs}
Let $\C$ be a Hom-finite Krull-Schmidt triangulated category over $K$ with a Serre functor $\mathbb{S}$. In this subsection, we study mutation of $n$-cotorsion pairs.

\begin{definition}\label{e}
Fix a functorially finite $(n+1)$-rigid subcategory $\D$ of $\C$. For a subcategory $\X$ of $\C$, put
$$\mu^{-1}_{\D}(\X):=(\D\ast\X[1])\cap\bigcap\limits_{i=1}^{n}{^\bot}\D[i].$$
That is, $\mu^{-1}_{\D}(\X)$ consists of all $M\in\C$ such that there exists a triangle
$$X\stackrel{f}\rightarrow D\rightarrow M\rightarrow X[1]$$
with $X\in\X$ and a left $\D$-approximation $f$. Dually, for a subcategory $\Y$ of $\C$, put
$$\mu_{\D}(\Y):=(\Y[-1]\ast\D)\cap\bigcap\limits_{i=1}^{n}\D[-i]^\perp.$$
That is, $\mu_{\D}(\Y)$ consists of all $M\in\C$ such that there exists a triangle
$$M\rightarrow D\stackrel{g} \rightarrow Y\rightarrow M$$
with  $Y\in\Y$ and a right $\D$-approximation  $g$.

 In this case, $\mu^{-1}_{\D}(\X)$ is called the forward $\D$-mutation of $\X$ and $\mu_{\D}(\Y)$ is called the backward $\D$-mutation of $\Y$.
\end{definition}

\begin{remark}
We also have that $\mu_{\D}(\D)=\D=\mu^{-1}_{\D}(\D)$. When $\D=0$, we have $\mu^{-1}_{\D}(\M)=\M[1]$ and $\mu_{\D}(\M)=\M[-1]$. When $n=1$, our definition of mutation is just Iyama-Yoshino's definition (see Definition \ref{h}). It's important to note that the concept of  mutation  mentioned later specifically refers to our definition, which is Definition \ref{e}.
\end{remark}

The following result was proved in \cite[Proposition 2.7]{IY} for the case $n=1$. But their proof
can be applied for the general case without any change.

\begin{proposition}
Let $\D$ be a functorially finite $(n+1)$-rigid subcategory of $\C$  satisfying  $\mathbb{S}\D=\D[n+1]$ and $\Z=\bigcap\limits_{i=1}^{n}\D[-i]^\perp=\bigcap\limits_{i=1}^{n}{^\bot}\D[i]$. Then the maps $\mu^{-1}_{\D}(-)$ and $\mu_{\D}(-)$ are mutually inverse on the set of subcategories $\M$ of $\C$ satisfying $\D\subset\M\subset\Z$.
\end{proposition}

Now we give the main result of this paper.

\begin{theorem}\label{main}
Let $(\X,\Y)$ be an $n$-cotorsion pair in $\C$ and $\D\subset\I(\X)$ be a functorially finite subcategory of $\C$ satisfying  $\mathbb{S}\D=\D[n+1]$. Then the pairs $$(\mu^{-1}_{\D}(\X),\mu^{-1}_{\D}(\Y))\hspace{2mm}\mbox{and}\hspace{3mm}(\mu_{\D}(\X),\mu_{\D}(\Y))$$ are  $n$-cotorsion pairs in $\C$ and we have
$$\I(\mu^{-1}_{\D}(\X))=\mu^{-1}_{\D}(\I(\X))\hspace{2mm}\mbox{and}\hspace{3mm}\I(\mu_{\D}(\X))=\mu_{\D}(\I(\X)).$$
\end{theorem}
\begin{proof}
By Remark \ref{l}, we have that $\I(\X)$ is $(n+1)$-rigid, then so is $\D$. So both $\mu^{-1}_{\D}(\X)$ and $\mu^{-1}_{\D}(\Y)$ are well-defined. We denote $\mu^{-1}_{\D}(\X)$, $\mu^{-1}_{\D}(\Y)$ and $\mu^{-1}_{\D}(\I(\X))$ by
$\X^\prime$,  $\Y^\prime$ and $\I^\prime$ respectively. As in the previous subsection, we denote by $\Z=\bigcap\limits_{i=1}^{n}\D[-i]^\perp=\bigcap\limits_{i=1}^{n}{^\bot}\D[i]$ and $\mathfrak{U}=\Z/\D$ the subfactor triangulated category. By Theorem \ref{d}, we have that $(\overline{\X},\overline{\Y})$ is an $n$-cotorsion pair in $\mathfrak{U}$ and $\I(\overline{\X})=\overline{\I(\X)}$. Note that its  forward 0-mutation $(\overline{\X}\langle1\rangle,\overline{\Y}\langle1\rangle)$ is also an $n$-cotorsion pair in $\mathfrak{U}$. By \cite[Proposition 4.4]{IY}, we have that $\overline{\X}\langle1\rangle=\overline{\X^\prime}$ and $\overline{\Y}\langle1\rangle=\overline{\Y^\prime}$. Then by Theorem \ref{d}, we have that the pair $(\X^\prime,\Y^\prime)$ is an $n$-cotorsion pair in $\C$ and $\I(\overline{\X^\prime})=\overline{\I(\X^\prime)}$. This means that $\I^\prime=\overline{\I(\X)}\langle1\rangle=\I(\overline{\X})\langle1\rangle=\I(\overline{\X})\langle1\rangle=
\I(\overline{\X^\prime})=\overline{\I(\X^\prime)}$. Therefore, $\I(\mu^{-1}_{\D}(\X))=\mu^{-1}_{\D}(\I(\X))$. The assertion for $\I(\mu_{\D}(\X))=\mu_{\D}(\I(\X))$ can be proved similarly.
\end{proof}
\section{A geometric realization of mutation of $n$-cotorsion pairs in $n$-cluster categories of type $A_m$}
In this section, we recall from \cite{BM,L,CZ2} a geometric description of $n$-cluster categories of type $A_m$, denoted by $\C_{A_{m}}^{n}$. Then we give a geometric characterization of $n$-cotorsion pairs in $\C_{A_{m}}^{n}$ via certain configurations of $n$-diagonals, and introduce the rotation of such combinatorial models, which can give a geometric realization of mutation of $n$-cotorsion pairs in $\C_{A_{m}}^{n}$.
\subsection{$n$-cluster categories of type $A_m$ and its geometric realization}
Let $H$ be a finite dimensional hereditary algebra over a field $K$, and $\H=\D^{b}(H)$ be the bounded derived category of $H$. The Auslander-Reiten translate of $\H$ is denoted by $\tau$ and the shift functor of $\H$ is denoted by $[1]$. Fix an integer $m\in\mathbb{Z}_{\geq 1}$ throughout this subsection.
\begin{definition}[\cite{K}]
The orbit category $\D^{b}(H)/\tau^{-1}[n]$ is called the $n$-cluster category of $\H$, and is denoted by $\C^{n}(\H)$.
\end{definition}
Specially, when $H=KQ$ is the path algebra over a Dynkin quiver of type $A_{m}$, the $n$-cluster category $\C^{n}(\H)$, denoted by $\C_{A_{m}}^{n}$, is called the $n$-cluster category of type $A_{m}$, which we consider in this section.

Now we summarize some known facts about $\C_{A_{m}}^{n}$ from \cite{BMRRT,K}.
\begin{proposition}
\begin{enumerate}
  \item The category $\C_{A_{m}}^{n}$ has Auslander-Reiten triangles and Serre functor $\mathbb{S}=\tau[1]$, where $\tau$ is the AR-translate and $[1]$ is the shift functor in $\C_{A_{m}}^{n}$.
  \item The category  $\C_{A_{m}}^{n}$ is a Krull-Schmidt and $(n+1)$-Calabi-Yau triangulated category.
\end{enumerate}
\end{proposition}
Since $\C_{A_{m}}^{n}$ is $(n+1)$-Calabi-Yau, $[n+1]$ is a Serre functor. Hence any subcategory $\D$ of $\C_{A_{m}}^{n}$ satisfies the condition $\mathbb{S}\D=\D[n+1]$. We have the following useful result.
\begin{lemma}\cite[Lemma 3]{T}\label{f}
Suppose $X, Y\in\C_{A_{m}}^{n}$. Then $$\mathrm{dim}_{K}\Ext^{i}_{\C_{A_{m}}^{n}}(X,Y)=\mathrm{dim}_{K}\Ext^{n+1-i}_{\C_{A_{m}}^{n}}(Y,X)$$
for all $i=1,\ldots,n$.
\end{lemma}
Let $\Pi$ be a regular $(n(m+1)+2)$-gon, $m,n\in\mathbb{Z}_{\geq 1}$, with vertices numbered clockwise from 1 to $n(m+1)+2$. We regard all operations on vertices of $\Pi$ modulo $n(m+1)+2$ in this section. A diagonal of $\Pi$ is a straight line between two non-adjacent vertices $i$ and $j$ of $\Pi$, denoted by $(i,j)$. Thus $(i,j)=(j,i)$.
\begin{definition}
A diagonal $\alpha$ is called an  $n$-diagonal  if $\alpha$ divides $\Pi$ into an $(nj+2)$-gon and $(nk+2)$-gon for some integers $j,k$.
\end{definition}
\begin{remark}
For a diagonal $\alpha=(i,j)$ of $\Pi$ with $i<j$, $\alpha$ is  an  $n$-diagonal  if and only if $j-i\equiv1\mod n$.
\end{remark}
There is a translation $\tau_{n}$ on the set of $n$-diagonals, which sends the $n$-diagonal $(i,j)$ to $(i-n,j-n)$, i.e. $\tau_{n}(i,j)=(i-n,j-n)$.
\begin{definition}(\cite{L,T})
Suppose $u=(c_1,c_2), v=(d_1,d_2)$ are two $n$-diagonals in $\Pi$ with $c_1<c_2, d_1<d_2$. The $n$-diagonals $u$ and $v$ are called crossing if their endpoints are all distinct and come in the order $c_1, d_1, c_2, d_2$ when moving around the polygon in one direction or the other, i.e. $c_1<d_1<c_2<d_2$ or $d_1<c_1<d_2<c_2$.
\end{definition}
By Definition above, the $n$-diagonals $u$ and $\tau_n(u)$ always cross for all $u$ in $\Pi$, see Lemma 3.1 in \cite{L}.
\begin{proposition}\cite[Proposition 5.4]{BM}
 There exists a bijection between $n$-diagonals in $\Pi$ and indecomposable objects in  $\mathrm{ind}\;\C_{A_{m}}^{n}$.
 \end{proposition}
The bijection above induces a bijection between the subcategories of $\C_{A_{m}}^{n}$ and the sets of $n$-diagonals in $\Pi$. For a subcategory $\X$ of $\C_{A_{m}}^{n}$, we denote the corresponding set of $n$-diagonals in $\Pi$ by $\XX$. For a set of $n$-diagonals $\mathfrak{X}$ in $\Pi$,  we denote the corresponding subcategory of $\C_{A_{m}}^{n}$ by $E(\mathfrak{X})$. We sometimes use an $n$-diagonal $(i,j)$ to represent an indecomposable object $M$ in $\C_{A_{m}}^{n}$ without confusion, denote by $M=(i,j)$.
\begin{example}
Suppose $n=2, m=3$. Then the AR-quiver of $\C_{A_{3}}^{2}$ is shown as Figure \ref{AR-quiver}.
\begin{figure}[h]
\centering

$$
{
\xymatrix@-7mm@C-0.17cm{
     &&(1,8)  \ar[rdd] & & (3,10) \ar[rdd] && (2,5) \ar[rdd]&& (4,7) \ar[rdd]&& (6,9)\ar[rdd] && (1,8) \\
 \\
 & (1,6)  \ar[rdd] \ar[ruu] & & (3,8) \ar[rdd]\ar[ruu] & & (5,10) \ar[rdd] \ar[ruu] & & (2,7) \ar[ruu]\ar[rdd]&& (4,9) \ar[ruu]\ar[rdd]&&(1,6) \ar[ruu]\\
 \\
  (1,4) \ar[ruu]&& (3,6) \ar[ruu] && (5,8) \ar[ruu]&&(7,10) \ar[ruu]&& (2,9) \ar[ruu]&& (1,4)\ar[ruu] && \\
}
}
$$
\caption{The AR-quiver of $\C_{A_{3}}^{2}$}
\label{AR-quiver}
\end{figure}
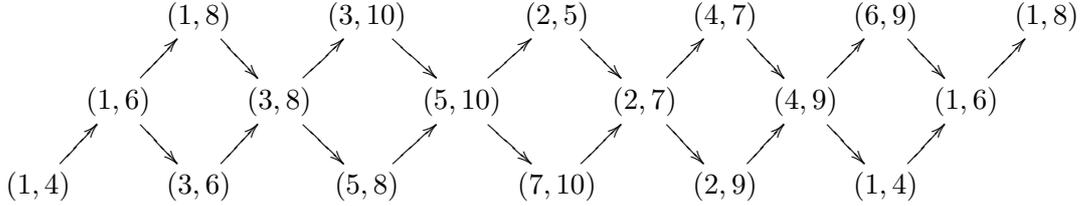
\end{example}
\begin{lemma}\cite[Proposition 2]{T}\label{g}
Suppose $u, v$ are two $n$-diagonals in $\Pi$, and $M_u, M_v$ are the corresponding indecomposable objects in $\C_{A_{m}}^{n}$. Then
\begin{itemize}
  \item [(1)] $u$ does not cross $v$ if and only if  $\Ext^{i}_{\C_{A_{m}}^{n}}(M_u,M_v)=0$ for all $1\leq i\leq n$.
  \item [(2)] $M_u[n]=M_{\tau_nu}$
\end{itemize}
\end{lemma}
\subsection{$n$-cotorsion pairs in $\C_{A_{m}}^{n}$}
In this subsection, we give a geometric characterization of $n$-cotorsion pairs in $\C_{A_{m}}^{n}$.

Let  $\mathfrak{X}$ be a set of $n$-diagonals in $\Pi$. We define
$$\nc\mathfrak{X}=\{u\in\Pi\text{\;is\;an\;} n\text{-diagonal\;}|\;u \text{\;does\;not\;cross\;any\;} n\text{-diagonals\; in\;}\mathfrak{X}\}.$$
By Lemma \ref{g}, $\nc\mathfrak{X}$ corresponds to the subcategory $\bigcap\limits_{i=1}^{n}\X[-i]^\perp=\bigcap\limits_{i=1}^{n}{^\bot}\X[i]$.
\begin{theorem}\label{i}
 Let $\X$ be a subcategory of $\C_{A_{m}}^{n}$, and $\XX$ be the corresponding set of $n$-diagonals in $\Pi$. Then the following statements are equivalent.
\begin{itemize}
  \item [(1)] $(\X,\bigcap\limits_{i=1}^{n}\X[-i]^\perp)$ is an $n$-cotorsion pair.
  \item [(2)] $\XX=\nc\nc\XX$.
  \end{itemize}
\end{theorem}
\begin{proof}
Since $\C_{A_{m}}^{n}$ is a finite triangulated category, any subcategory $\X$ of $\C_{A_{m}}^{n}$ is functorially finite. The equivalence is obvious by Lemma \ref{g}.
\end{proof}
We recall from \cite{CZ2} the definition of Ptolemy diagram and give a necessary condition for an $n$-cotorsion pair.
\begin{definition}
Let $\mathfrak{X}$ be a set of $n$-diagonals in $\Pi$. Then $\mathfrak{X}$ is called a Ptolemy diagram if for any two crossing $n$-diagonals $u=(i,j)$ and $v=(k,\ell)$ in $\mathfrak{X}$, those of $(i,k),(i,\ell),(j,k),(j,\ell)$, which are $n$-diagonals are in $\mathfrak{X}$.
\end{definition}
\begin{lemma}\label{j}
Let $\X$ be a subcategory of $\C_{A_{m}}^{n}$, and $\XX$ be the corresponding set of $n$-diagonals in $\Pi$. If $\XX=\nc\nc\XX$, then $\XX$ is a Ptolemy diagram.
\end{lemma}
\begin{proof}
By Theorem \ref{i}, it is enough to show for any subcategory $\Y$ of $\C_{A_{m}}^{n}$, $\nc\YY$ is a Ptolemy diagram.

Suppose $u=(i,j)$ and $v=(k,\ell)$ are two crossing $n$-diagonals in $\nc\YY$, and suppose $(i,k)$ is an $n$-diagonal, we need to show $(i,k)\in\nc\YY$.

If $(i,k)\not\in\nc\YY$, then there exists an $n$-diagonal $w\in\YY$ such that $w$ crosses $(i,k)$. Oberserve that every $n$-diagonal $w\in\YY$ crossing $(i,k)$ crosses either $u$ or $v$,  which is a contradiction with the assumption $u,v\in\nc\YY$.
\end{proof}

Note that when $n=1$, the converse of Lemma \ref{j} is true by \cite[Proposition 2.7]{HJR1}. However, the converse of Lemma \ref{j} is not true in general case. We see an example below.
\begin{example}\label{z}
Suppose $n=2, m=3$. The AR-quiver of $\C_{A_{3}}^{2}$ has been shown in Figure \ref{AR-quiver}.
\begin{itemize}
  \item [(1)] Let $\X=\{(1,4),(1,6),(3,6)\}$ be a subcategory of $\C_{A_{3}}^{2}$. Then $\XX$ is a Ptolemy diagram. However, $\nc\XX=\{(1,6),(1,8),(7,10),(6,9)\}$ and $\nc\nc\XX=\{(1,4),(1,6),(3,6),(2,5)\}$, so we have $\XX\neq\nc\nc\XX$. Thus, $(\X,\X[-1]^\perp\cap\X[-2]^\perp)$ can not be a $2$-cotorsion pair.
  \item [(2)] Let $\Y=\{(1,4),(1,6),(3,6),(2,5)\}$ be a subcategory of $\C_{A_{3}}^{2}$. We have $\nc\YY=\{(1,6),(1,8),(7,10),(6,9)\}$ and $\nc\nc\YY=\{(1,4),(1,6),(3,6),(2,5)\}$, so we have $\YY=\nc\nc\YY$. Thus, $(\Y,\Y[-1]^\perp\cap\Y[-2]^\perp)$ is a $2$-cotorsion pair.
  \end{itemize}

\end{example}
\subsection{Mutation of $n$-cotorsion pairs in $\C_{A_{m}}^{n}$}
\begin{definition}
Let $\X$ be a subcategory of $\C_{A_{m}}^{n}$ and $\XX$ be the corresponding set of $n$-diagonals in $\Pi$ satisfying $\XX=\nc\nc\XX$. The set of $n$-diagonals in $\XX$ that does not cross any diagonals in $\XX$ is called the frame of $\XX$, denoted by $F_{\XX}$.
\end{definition}
By the assumption above, the subcategory $\X$ is a first half of an $n$-cotorsion pair in $\C_{A_{m}}^{n}$ by Theorem \ref{i}, and the subcategory $E(F_{\XX})$ which corresponds to $F_{\XX}$ is the core of the $n$-cotorsion pair $(\X,\bigcap\limits_{i=1}^{n}\X[-i]^\perp)$, i.e. $\I(\X)=E(F_{\XX})$.
\begin{definition}
\begin{itemize}
\item [(1)] For any $n$-diagonal $(i,j)$ of $\Pi$, its rotation in $\Pi$, denoted by $\rho_{\Pi}(i,j)$, is defined to be $(i-n,j-n)$, i.e.
      $$\rho_{\Pi}(i,j)=(i-n,j-n).$$
\item [(2)] Let $\mathfrak{D}$ be a set of non-crossing $n$-diagonals of $\Pi$.
\begin{itemize}
  \item [(i)] The $n$-diagonals in $\mathfrak{D}$ divide $\Pi$ into polygons, called $\mathfrak{D}$-cells. Thus, any $n$-diagonals $(i,j)$ which neither is in $\mathfrak{D}$ nor crosses any diagonals in $\mathfrak{D}$ is an $n$-diagonal of a $\mathfrak{D}$-cell. We call the rotation of $(i,j)$ in this $\mathfrak{D}$-cell the $\mathfrak{D}$-rotation of $(i,j)$, and denote it by $\rho_{\D}(i,j)$.
  \item [(ii)] Let $\mathfrak{X}$ be a set of diagonals of $\Pi$ satisfying $\mathfrak{X}=\nc\nc\mathfrak{X}$. The $\mathfrak{D}$-rotation of $\mathfrak{X}$ is defined as
      $$\rho_{\mathfrak{D}}(\mathfrak{X}):=\{\rho_{\mathfrak{D}}(i,j)|(i,j)\in\mathfrak{X}\backslash\mathfrak{D}\}\cup\mathfrak{D}.$$
\end{itemize}
\end{itemize}
\end{definition}
\begin{example}
Suppose $n=2, m=3$. The AR-quiver of $\C_{A_{3}}^{2}$ has been shown in Figure \ref{AR-quiver}. Let
$$\mathfrak{X}=\{(1,4),(1,6),(1,8),(7,10),(6,9)\}.$$
It is easy to check the set $\mathfrak{X}$ satisfies $\mathfrak{X}=\nc\nc\mathfrak{X}$ and $F_{\mathfrak{X}}=\{(1,4),(1,6)\}$. Let $\mathfrak{D}=\{(1,6)\}\subset F_{\mathfrak{X}}$. Then $\mathfrak{D}$ divides $\Pi$ into two polygons. The $\mathfrak{D}$-rotation of $\mathfrak{X}$ is
$$\rho_{\mathfrak{D}}(\mathfrak{X})=\{(2,5),(1,6),(6,9),(1,8),(7,10)\},$$
where for example $\rho_{\mathfrak{D}}(1,4)=(2,5).$ Let $\mathfrak{Y}=\rho_{\mathfrak{D}}(\mathfrak{X})$. It is clear that $\mathfrak{Y}$ still satisfies $\mathfrak{Y}=\nc\nc\mathfrak{Y}$.

See Figure \ref{8} and  Figure \ref{9} for  geometric explanations.
\end{example}
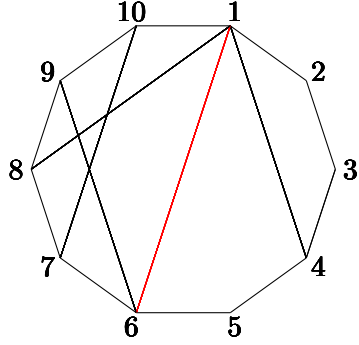
\begin{figure}[h]
\centering
\begin{tikzpicture}[scale=2]
\foreach \x in {-36,0,36,72,108,144,180,216,252,288,324}{
\draw (\x:1 cm) -- (\x + 36: 1cm) -- cycle;
\draw (36*1+36:1.1) node {$1$} ;
\draw (36*0+36:1.1) node {$2$} ;
\draw (36*2+36:1.1) node {$10$} ;
\draw (36*3+36:1.1) node {$9$} ;
\draw (36*4+36:1.1) node {$8$} ;
\draw (36*5+36:1.1) node {$7$} ;
\draw (36*6+36:1.1) node {$6$} ;
\draw (36*7+36:1.1) node {$5$} ;
\draw (36*8+36:1.1) node {$4$} ;
\draw (36*9+36:1.1) node {$3$} ;
\draw (36*9:1cm)--(36*2:1cm);
\draw[red] (36*7:1cm)--(36*2:1cm);
\draw (36*5:1cm)--(36*2:1cm);
\draw (36*6:1cm)--(36*3:1cm);
\draw (36*7:1cm)--(36*4:1cm);
}
\end{tikzpicture}
\caption{The set $\mathfrak{X}$ satisfies $\mathfrak{X}=\nc\nc\mathfrak{X}$ with $\mathfrak{D}=\{(1,6)\}$}
\label{8}
\end{figure}
\begin{figure}[h]
\centering
\begin{tikzpicture}[scale=2]
\foreach \x in {-36,0,36,72,108,144,180,216,252,288,324}{
\draw (\x:1 cm) -- (\x + 36: 1cm) -- cycle;
\draw (36*1+36:1.1) node {$1$} ;
\draw (36*0+36:1.1) node {$2$} ;
\draw (36*2+36:1.1) node {$10$} ;
\draw (36*3+36:1.1) node {$9$} ;
\draw (36*4+36:1.1) node {$8$} ;
\draw (36*5+36:1.1) node {$7$} ;
\draw (36*6+36:1.1) node {$6$} ;
\draw (36*7+36:1.1) node {$5$} ;
\draw (36*8+36:1.1) node {$4$} ;
\draw (36*9+36:1.1) node {$3$} ;
\draw (36*8:1cm)--(36*1:1cm);
\draw[red] (36*7:1cm)--(36*2:1cm);
\draw (36*5:1cm)--(36*2:1cm);
\draw (36*6:1cm)--(36*3:1cm);
\draw (36*7:1cm)--(36*4:1cm);
}
\end{tikzpicture}
\caption{The set $\mathfrak{Y}=\rho_{\mathfrak{D}}(\mathfrak{X})$ still satisfies $\mathfrak{Y}=\nc\nc\mathfrak{Y}$}
\label{9}
\end{figure}
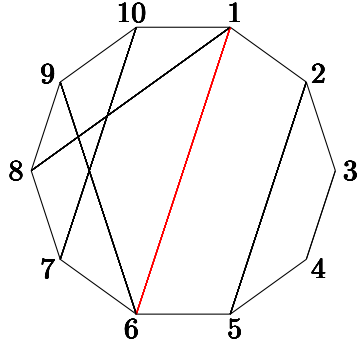
Now we prove that rotation of a set of $n$-diagonals $\mathfrak{X}$ of $\Pi$ satisfying $\mathfrak{X}=\nc\nc\mathfrak{X}$ gives a geometric model of mutation of $n$-cotorsion pairs in $\C_{A_{m}}^{n}$.

\begin{theorem}\label{mainresult}
Let $\X$ be a subcategory of $\C_{A_{m}}^{n}$,  $\XX$ be the corresponding set of $n$-diagonals in $\Pi$ satisfying $\XX=\nc\nc\XX$, and $F_{\XX}$ be the frame of $\XX$. Then the following hold.
\begin{itemize}
  \item [(1)] The frame $F_{\XX}$ of $\XX$ corresponds to the core of the $n$-cotorsion pair \\ $(\X,\bigcap\limits_{i=1}^{n}\X[-i]^\perp)$, i.e. $\I(\X)=E(F_{\XX})$, where $E(F_{\XX})$ represents the subcategory of $\C_{A_{n}}^{m}$ which corresponds to $F_{\XX}$.
   \item [(2)] Let $\mathfrak{D}\subset F_{\XX}$ be a set of non-crossing $n$-diagonals and $E(\mathfrak{D})$ be the corresponding subcategory of $\C_{A_{m}}^{n}$,  we have that
       $$\mu_{E(\mathfrak{D})}(\X)=E(\rho_{\mathfrak{D}}(\XX)).$$
\end{itemize}
Consequently, a rotation of a set of  $n$-diagonals $\mathfrak{X}$ of $\Pi$ satisfying $\mathfrak{X}=\nc\nc\mathfrak{X}$  is again a set of $n$-diagonals $\mathfrak{Y}=\rho_{\mathfrak{D}}(\mathfrak{X})$ of $\Pi$ satisfying $\mathfrak{Y}=\nc\nc\mathfrak{Y}$.
\end{theorem}

\begin{proof}
It is obvious that (1) holds.

In order to prove (2), it is enough to show that for any $n$-diagonal $(i,j)(i<j)$ in $\XX$ which is not in $\mathfrak{D}$, if $\rho_{\mathfrak{D}}(i,j)=(k,\ell)$, then $\mu_{E(\mathfrak{D})}(E(i,j))=E(k,\ell)$.

Let $\mathfrak{X}^\prime$ be the $\mathfrak{D}$-cell containing $(i,j)$. Since $\rho_{\mathfrak{D}}(i,j)=(k,\ell)$, $(k,\ell)$ is the rotation of $(i,j)$ in $\mathfrak{U}^\prime$. That means $(k,\ell)$ is an $n$-diagonal of $\mathfrak{X}^\prime$, so $(k,\ell)$ does not cross any diagonal in $\mathfrak{D}$, we have $E(k,\ell)\in\bigcap\limits_{i=1}^{n}E(\mathfrak{D})[-i]^\perp$.

Since $(i,j)$ is an $n$-diagonal, $(k,\ell)$ crosses $(i,j)$. Without loss of generality, we assume $k<i<\ell<j$. By the definition of rotation, we have $(k,\ell)=(i-n,j-n)$. By the AR-quiver of $\C_{A_{m}}^{n}$, it is easy to check that the crossing diagonals $(i,j)$ and $(k,\ell)$ induce a triangle in $\C_{A_{m}}^{n}$
$$E(k,\ell)\rightarrow E(i,\ell)\oplus E(k,j)\stackrel{f}\rightarrow E(i,j)\rightarrow E(k,\ell)[1]$$
By the definition of rotation, $(i,\ell)$ and $(k, j)$ are edges of $\mathfrak{X}^\prime$. Thus, $E(i,\ell)$ (resp. $E(k,j)$) is an zero object or a non-zero object in $E(\mathfrak{D})$, i.e. $E(i,\ell)\oplus E(k,j)\in E(\mathfrak{D})$. Applying $\Hom(D,-)$ to the above triangle for any $D\in E(\mathfrak{D})$, we have an exact sequence
$$\Hom(D,E(i,\ell)\oplus E(k,j))\rightarrow \Hom(D,E(i,j))\rightarrow \Hom(D,E(k,\ell)[1]),$$
where $\Hom(D,E(k,\ell)[1])=0$ since $E(k,\ell)\in\bigcap\limits_{i=1}^{n}E(\mathfrak{D})[-i]^\perp$. This implies that $f$ is a right $E(\mathfrak{D})$-approximation and it is minimal since $E(k,\ell)$ is indecomposable. So by definition, $\mu_{E(\mathfrak{D})}(E(i,j))=E(k,\ell)$ and we complete the proof of $\mu_{E(\mathfrak{D})}(\X)=E(\rho_{\mathfrak{D}}(\XX)).$ Then the last assertion follows from Theorem \ref{main}.
\end{proof}

\bigskip

%$\bold {Acknowlegement}$: We would like to thank the referees for their helpful comments and suggestions.

\end{document}